\documentclass[12pt]{article}
\usepackage{epstopdf}
\usepackage[labelformat=empty]{caption}

\usepackage{amsmath,amssymb}
\usepackage{amsthm}
\usepackage{graphicx,tikz}
\usepackage{enumerate}

\numberwithin{equation}{section}
\newtheorem{Thm}{Theorem}[section]
\newtheorem{Prop}[Thm]{Proposition}
\newtheorem{Lem}[Thm]{Lemma}
\newtheorem{Cor}[Thm]{Corollary}
\newtheorem{Fact}[Thm]{Fact}
\newtheorem{Exa}[Thm]{Example}
\newtheorem{LemDef}[Thm]{Lemma-Definition}

\theoremstyle{remark}

\newtheorem{Rem}[Thm]{Remark}

\theoremstyle{definition}

\newtheorem{Def}[Thm]{Definition}
\newtheorem*{Def*}{Definition}

\numberwithin{equation}{section}

\newcommand{\g}[1]{{\mbox{\goth #1}}}
\newcommand{\m}[1]{\mathbb{ #1}}
\newcommand{\mc}[1]{\mathcal{ #1}}
\newcommand{\gs}[1]{{\mbox{\gots #1}}}

\newcommand{\Ind}{{\operatorname{Ind}}}
\newcommand{\Tr}{{\operatorname{Trace}}}

\def\al{\alpha}       \def\be{\beta}        \def\ga{\gamma}
       \def\eps{\varepsilon}  
       \def\la{\lambda}      
\def\si{\sigma}                
\def\ph{\varphi}               
              \def\De{\Delta}

\newcommand{\rmd}{{\,\rm d}}

\theoremstyle{definition}

\theoremstyle{remark}

\newtheorem{Rmq}[Thm]{Remark}

\numberwithin{equation}{section}
\newfont{\goth}{eufm10 at 12pt}
\newfont{\gots}{eufm8 at 9pt}

\def\bt{\begin{Thm}}
\def\et{\end{Thm}}
\def\br{\begin{Rmq}}
\def\er{\end{Rmq}}

\def\bc{\begin{Cor}}
\def\ec{\end{Cor}}
\def\bp{\begin{Prop}}
\def\ep{\end{Prop}}
\def\bl{\begin{Lem}}
\def\el{\end{Lem}}
\def\bd{\begin{Def}}
\def\ed{\end{Def}}
\def\bq{\begin{quotation}}
\def\eq{\end{quotation}}
\def\bfa{\begin{Fact}}
\def\efa{\end{Fact}}

\def\ra{\rightarrow}
\def\vs{\vspace{1em}}

\begin{document}
\title{
Tempered homogeneous spaces II
}
\author{Yves Benoist and Toshiyuki Kobayashi
}
\date{}
\maketitle
\vspace{-2em}
\centerline{\footnotesize 
\mbox{ }\hfill Dedicated to G. Margulis.}

\begin{abstract}{
\noindent Let $G$ be a connected semisimple Lie group with finite center  and $H$ a connected closed subgroup.
We establish a geometric criterion which detects whether  the  representation  of $G$
in $L^2(G/H)$ is tempered.
}\end{abstract}
{\footnotesize \tableofcontents}

\section{Introduction}
\label{secintrod}

This article 
is the sequel of our paper \cite{BeKoI} dealing with  harmonic analysis 
on homogeneous spaces $G/H$ of semisimple Lie groups $G$. 
In the first paper \cite{BeKoI} we studied the regular representation 
of $G$ in $L^2(G/H)$ when both $G$ and $H$ are semisimple groups.
The main result of \cite{BeKoI} is a geometric criterion which detects whether  
the representation of $G$ in $L^2(G/H)$ is tempered. 
The aim of the present paper is to extend this geometric criterion 
to the whole generality that $H$ is an arbitrary closed connected subgroup.

In this introduction we will discuss the following questions:
\begin{itemize}
\item Why to care about tempered representations of semisimple Lie groups?
\item What is our temperedness criterion for the homogeneous space $G/H$?
\item What are the main ideas and ingredients in the proof of this criterion? 
\end{itemize}

\subsection{Tempered representations}
\label{sectemrep}
\hspace*{0em}\let\thefootnote\relax\footnote{\hspace*{-2em}
Key words: Lie groups, homogeneous spaces, tempered 
representations, matrix coefficients,
unitary representations\\
\textit{MSC (2010):}\enspace
Primary 22E46; 
Secondary 43A85,
22F30. }
Let $G$ be a semisimple Lie group with finite center
and $K$ a maximal compact subgroup of $G$.
Understanding unitary representations of $G$
in Hilbert spaces is a major topic 
of research since the beginning of the $20^{\rm th}$ century.
Its history includes the works of Cartan, Weyl, Gelfand,
Harish-Chandra, Helgason, Langlands, Vogan and many others.
The main motivations came from quantum physics, analysis and number theory.

Among unitary representations of $G$ a smaller class called 
{\it tempered representations} plays a crucial role. 
Let us recall why they are so useful.
\begin{itemize}
\item By definition  tempered representations are those 
which are weakly contained 
in the regular representation of $G$ in $L^2(G)$.
Therefore a unitary representation of $G$ is tempered 
if and only if its disintegration into irreducible 
unitary representations involves only 
tempered representations.
\item Tempered representations are those for which $K$-finite matrix coefficients belong to $L^{2+\eps}(G)$ for all $\eps>0$. This definition does not look so enlightening
but equivalently these matrix coefficients are bounded by an explicit 
multiple of an explicit spherical function $\Xi$, see \cite{CHH}.\
\item Classification of irreducible tempered representations of $G$
was accomplished by
Knapp and Zuckerman in \cite{KnZu82},
while non-tempered irreducible unitary representations 
have not yet been completely understood.
\item Tempered representations are a cornerstone of the Langlands classification 
of admissible irreducible representations of $G$
in \cite{Lan89}, see also \cite{Kn01}.
\item Irreducible tempered representations $\pi$
can be characterized in term of the {\it leading exponents}: these  exponents
must be a positive linear combination of negative roots, see \cite[Chap.~8]{Kn01}.
\item One can also characterize them in terms of the 
{\it distribution character of $\pi$}: this  character must
be a tempered distribution on $G$, 
see \cite[Chap.~12]{Kn01}. 
\item Tempered representations are closed under induction, restriction,
tensor product, and direct integral of unitary representations. 
\item The Kirilov--Kostant orbit methods works fairly well for tempered representations,
 see \cite{HaTo} for example. 
\end{itemize}

\subsection{The regular representation
in $L^2(G/H)$}
\label{secregrep}

One of the most studied representations of $G$ is
the natural unitary representations of $G$ in $L^2(G/H)$
where 
$H$ is  a  closed subgroup of $G$.
When $H$ is unimodular, the space $G/H$ is implicitly endowed with 
a $G$-invariant measure and $G$ acts naturally by translation on 
$L^2$-functions. When $H$ is not assumed to be unimodular,
the representation of $G$ in $L^2(G/H)$ might involve an extra factor \eqref{eqnregrec} 
and is nothing but the (unitarily) induced representation
$Ind_H^G(1)$.  

The disintegration of $L^2(G/H)$
is sometimes called the {\textit{Plancherel formula}} or $L^2$-harmonic analysis on $G/H$. 
\begin{itemize}
 \item 
For instance, Harish-Chandra's celebrated Plancherel formula \cite{HaCh76}
deals with  the case where $H=\{ e\}$.
\item
Another case which attracted a lot of interest in the 
late $20^{\rm th}$ century is the case 
where $G/H$ is a symmetric space,
for which the disintegration of $L^2(G/H)$ is proved 
up to the classification of (singular) discrete series representations 
for symmetric spaces \cite{Del98, Osh88}.
\item 
Even when $H$ is not unimodular, 
the regular representation of $G$ in $L^2(G/H)$ is still interesting.
For instance,  when $H$ is a parabolic subgroup of $G$, this representation is known to be
a finite sum of irreducible representations of $G$.
This follows from Bruhat's theory for minimal parabolic and is due to Harish-Chandra in general
(see \cite[Prop. 8.4]{Kn01}).
\item
The decomposition of the tensor product ({\textit{fusion rule}}) is sometimes equivalent to
 the Plancherel formula for $G/H$ where $H$ is not necessarily unimodular
(see Section \ref{sectensor}).
\end{itemize}

We refer to  \cite[Intro.]{BeKoI} for more remarks on the
historical developments of the disintegration of the regular representation 
of $G$ in $L^2(G/H)$.
Getting {\textit {a priori}} information on this disintegration 
was one of the main motivations in our search for such a general criterion.
To the best of our knowledge there does not exist yet any general theorem involving
simultaneously 
all these regular representations in $L^2(G/H)$ with $H$ connected.
Our temperedness criterion below seems to be the first one in that direction.
\vs

In this series of papers, we address the following question: 
\textsl{What kind of unitary representations occur
in the disintegration of $L^2(G/H)$?} More precisely, 
\textsl{when are all of them tempered?}
\vs

As noted in \cite{BeKoI}, this question has not been completely solved even for
 reductive symmetric spaces $G/H$, because the Plancherel formula 
involves a delicate algebraic problem on 
 discrete series representations for (sub)sym\-metric spaces
 with singular infinitesimal characters (see Example~\ref{exSp}). 

We give a geometric necessary and sufficient condition  
on $G/H$   under which all these irreducible unitary representations 
in the Plancherel formula  are tempered, 
or equivalently under which the regular representation of $G$ in $L^2(G/H)$ is tempered.
This criterion was first discovered in our paper \cite{BeKoI}
in the special case where $H$ is a reductive subgroup of $G$.

In the present paper we extend this criterion 
to any  closed subgroup $H$ with finitely many connected components.

Formally the extended criterion is exactly the same as for $H$ reductive.
Here is a short way to state our criterion (see Theorem \ref{thmlghtem}).
\begin{equation}
\label{eqnlghtem}
 L^2(G/H)\; \mbox{\rm is tempered}
\;\;\;\Longleftrightarrow \;\;\;
\rho_{\gs h}(Y)\leq \rho_{\gs g/\gs h}(Y)
\;\;\;\mbox{\rm for all $Y\in \g h$.}
\end{equation}
Here  $\g g$ and $\g h$ are the Lie algebras of $G$ and $H$, 
and, for an $\g h$-module $V$ and $Y\in \g h$, the quantity
$\rho_V(Y)$ is
half the sum of the absolute values of the real part of the eigenvalues of $Y$ in $V$ (see Section \ref{secfunrov}).
We note that our criterion holds beyond  the (real) spherical case (\cite{ko13fm},
cf.~\cite{SaVe17} for  non-archimedean field) 
so that  the disintegration of $L^2(G/H)$ may involve infinite multiplicities.  
We will give an explicit example of calculations of functions $\rho_V$ in 
Corollary \ref{corslpqr}.

\medskip
Our criterion \eqref{eqnlghtem} detects also whether or not 
$L^2(X)$ is tempered for any real algebraic $G$-variety $X$:
$L^2(X)$ is unitarily equivalent to 
the direct integral of the regular representations for generic orbits,
and one just has to check \eqref{eqnlghtem} at almost all orbits.

\subsection{Strategy of proof}
\label{secstrpro}

The proof of \eqref{eqnlghtem} 
relies on the uniform decay of matrix coefficients
as in \cite{BeKoI}, but the main techniques are  different from those  in \cite{BeKoI}.

 To avoid any confusion we will sometimes say {\it $G$-tempered} 
for {\it tempered as a representation of $G$}.

Dealing with non-unimodular subgroups $H$ and dealing with the finitely many components of $H$ will not be a problem 
because one proves in Corollary \ref{corgh1h2}  the equivalence
\begin{eqnarray*}
\label{eqnlghghh}
L^2(G/H)\; \mbox{\rm is tempered}
&\;\;\;\Longleftrightarrow \;\;\;&
L^2(G/[H,H])\; \mbox{\rm is tempered}\\
&\;\;\;\Longleftrightarrow \;\;\;&
L^2(G/H_e)\; \mbox{\rm is tempered}
\; ,
\end{eqnarray*}
where $[H,H]$ is the derived subgroup which is always unimodular 
and where $H_e$ is the identity component of $H$ which is always connected! 
Therefore the temperedness of $L^2(G/H)$ depends only on $\g h$ and 
we can assume that $\g h=[\g h,\g h]$. Then 
the homogeneous space $G/H$ admits a $G$-invariant Radon measure ${\rm vol}$ and, 
according to Corollary \ref{coralpcol}, the  temperedness of $L^2(G/H)$ means that the 
compact subsets $C$ of $G/H$ intersect their translates $gC$ in a set whose  volume is bounded 
by a multiple  of the Harish-Chandra function $\Xi$:
\begin{equation}
\label{eqnvolxig}
{\rm vol}(g\, C\cap C) \leq M_C\,\Xi(g)
\;\;\;\; \mbox{\rm for all $g$ in $G$.}
\end{equation}

To prove the direct implication 
``$L^2(G/H)$ tempered $\Longrightarrow$ $\rho_{\gs h}\leq \rho_{\gs g/\gs h}$'',
we will just estimate in Proposition \ref{prolghtem} this volume \eqref{eqnvolxig} when $C$ is a small neighborhood
of the base point of the space $G/H$.

As in \cite{BeKoI}, the converse implication  
``$\rho_{\gs h}\leq \rho_{\gs g/\gs h}$ $\Longrightarrow$ $L^2(G/H)$ tempered'',
is much harder to prove. 
Using again Corollary \ref{corgh1h2},  we can also assume that both $G$ and $H$ 
are Zariski connected algebraic groups. 
We proceed by induction on the dimension of $G$.
We introduce in Definition-Lemma \ref{defparsub} 
two intermediate subgroups $F$ and $P$:
\begin{equation}
\label{eqnFP}
          H \subset F \subset P \subset G\, .
\end{equation}
The group $P$ is a parabolic subgroup of minimal dimension that contains $H$.
We write $P=LU$ and $H=SV$ where $U$ is the unipotent radical of $P$,
$L$ is a maximal reductive subgroup of $P$,
$V\subset U$ is the unipotent radical of $H$,
and $S\subset L$ a maximal semisimple subgroup of $H$.
The group $F$ is given by $F:=SU$.

When $P$ is equal to $G$, the group $H$ is semisimple and we apply the main result 
of \cite{BeKoI}. We now assume that $P$ is a proper parabolic subgroup of $G$. 
Let $Z$ be the homogeneous space $Z=F/H\simeq U/V$ endowed with the natural $F$-action.
We denote\\ 
- by $\tau$ the regular representation of $F$ in $L^2(F/H)$,\\
- by  $\pi$ the regular representation of $P$ in $L^2(P/H)$,\\ 
- by $\Pi$ the regular representation of $G$ in $L^2(G/H)$.

Let $Z_0$ be the  same space $Z_0= U/V$ but endowed with another $F$-action, 
where $U$ acts trivially
and where $S$ acts by conjugation.

We denote\\
- by $\tau_0$ the regular representation of $F$ in $L^2(Z_0)$,\\
- by  $\pi_0$ the regular representation of $P$ in $L^2(P\times_FZ_0)$,\\ 
- by $\Pi_0$ the regular representation of $G$ in $L^2(G\times_FZ_0)$.

Note that both the $G$-manifolds $X:=G/H \simeq G \times_F Z$ 
and $X_0 := G \times_F Z_0$ have 
a $G$-equivariant fiber bundle structure 
over the same space $G/F$ with fiber $Z \simeq Z_0$, 
but that the unitary representations of $G$ 
in $L^2(X)$ and in $L^2(X_0)$ are different. 
We want to study the representation of $G$ in  $L^2(X)$, 
whereas the induction hypothesis will give us information on  $L^2(X_0)$.
This is why we will need to  compare in \eqref{eqnvolvolf} below
the volumes in the fibers $Z$ and $Z_0$.

More precisely, the induction hypothesis combined with a reformulation
of our criterion in Proposition \ref{proequthm}
and a simple computation in Lemma \ref{lemparsub}  tell us that the representation $\pi$ is $L$-tempered.
If we knew that $\pi$ were a $P$-tempered representation 
it would be easy to conclude, using Lemma \ref{lemindtem}, that the representation 
$\Pi=\Ind_P^G\pi$ is $G$-tempered. 
However, what we know is the temperedness of $\pi_0$ and $\Ind_P^G \pi_0$, and
 Corollary \ref{coralpcol} gives us,
for any  
compact subset $C_0$ of $G\times_FZ_0$, 
a bound:
\begin{equation}
\label{eqnvolxi0}
{\rm vol}(g\, C_0\cap C_0) \leq M_{C_0}\,\Xi(g) 
\;\;\;\; \mbox{\rm for all $g$ in $G$.}
\end{equation}
In order to deduce \eqref{eqnvolxig} from \eqref{eqnvolxi0},
we first focus on the representations $\tau$ in $L^2(Z)$ and $\tau_0$ in $L^2(Z_0)$. 
We prove in Proposition \ref{prodom},
for every compact subset $D$ of the fiber $Z= F/H$, a uniform
 estimate of ${\rm vol}(f D \cap D)$ 
with respect to translates of the element $f\in F$ by elements of the unipotent subgroup $U$.
Namely,  there exists a compact subset 
$D_0\subset Z_0$ such that 
\begin{equation}
\label{eqnvolvolf}
{\rm vol}(fD\cap D) \leq {\rm vol}(fD_0\cap D_0)
\;\;\;\; \mbox{\rm for all $f$ in $F$.}
\end{equation}
Since $\Pi=\Ind_F^G \tau$ and $\Pi_0=\Ind_F^G \tau_0$, 
we deduce  from \eqref{eqnvolvolf} in Proposition \ref{provolgcc} that 
for every compact subset $C\subset G/H \simeq G\times_F Z$, there exists a compact subset 
$C_0\subset G\times_FZ_0$ such that 
\begin{equation}
\label{eqnvolvolg}
{\rm vol}(g\, C\cap C) \leq {\rm vol}(g\, C_0\cap C_0)
\;\;\;\; \mbox{\rm for all $g$ in $G$.}
\end{equation}
And \eqref{eqnvolxig} follows 
from \eqref{eqnvolxi0} and \eqref{eqnvolvolg}. This ends the
sketch of the  proof 
of the criterion \eqref{eqnlghtem}. 
The details are explained below. 

\subsection{Organization}
\label{secorg}

Here is the organization of the paper.

In Section \ref{secdefres} we recall the basic definition
and state precisely our criterion.

In Sections \ref{secprepro} we collect the parts 
of the proof which do not involve the intermediate parabolic 
subgroup $P$. 
It includes the proof for the necessity of the inequality $\rho_{\gs h}\leq \rho_{\gs g/\gs h}$,
 and a formulation of the
 temperedness criterion for $\Ind_H^G(L^2(V))$ (Theorem~\ref{thmlgvtem}).

In Section \ref{secparsub} we introduce the intermediate subgroups $F$ and $P$ in \eqref{eqnFP},
and detail how they work to conclude the proof for
the hard part, {\textit{i.e.}}, the proof that  the inequality $\rho_{\gs h}\leq \rho_{\gs g/\gs h}$
is sufficient.

In Section \ref{secexalgh} we give a few examples which illustrate the efficiency 
of our criterion. 
\vs

\subsection{Dedication}
\label{secded}

The notion of tempered representations plays
an important role in quite a few papers of G.Margulis.
As  we have seen the tempered representations is
one of the main tools
in order to obtain uniform estimates for coefficients
of unitary representations.
G. Margulis found many very different and very ingenious applications of
these uniform estimates all along his mathematical career. 

For instance to the construction of expanders 
in 
\cite{Ma73}, to the non-existence of compact quotient of homogeneous spaces in \cite{Ma97},
to  a pointwise ergodic theorem for probability measure preserving actions of semisimple Lie groups in \cite{MaNeSt}, to 
a local rigidity phenomenon for actions of higher rank lattices in \cite{FiMa},
to an effective rate of equidistribution of closed orbits of semisimple Lie groups in finite volume spaces in \cite{EiMaVe},   to a uniform estimate on the smallest integral base change
between two equivalent non-singular integral quadratic forms in  \cite{LiMa16}, \ldots

One  inspiration for  Theorem \ref{thmlghtem}  
came from  Lemma 6.5.4 of
\cite{EiMaVe} which says
that for a connected semisimple Lie group $G$ with no compact factor
and a proper closed connected subgroup $H$,
the representation
of $G$ in $L^2(G/H)$ always has a spectral gap.
\vs

We are proud to dedicate this paper to G. Margulis.

\section{Definition and main result}
\label{secdefres}

We collect in this chapter a few well-known facts  on tempered representations, on almost $L^2$ represen\-ta\-tions,
and on uniform decay of matrix coefficients.

\subsection{Regular and induced representations}
\label{secregind}
\bq
We first recall the construction of 
the regular representations and the 
induced representations. 
\eq

\subsubsection{Regular representations}
Let $G$ be a separable locally compact group,
$X$ be a  separable locally compact space endowed with a continuous action of $G$
and let $\nu_X$ be a Radon measure on $X$.

When the $G$-action preserves the measure $\nu_X$, 
one has a natural unitary representation $\la_X$ of $G$ in the 
Hilbert space $L^2(X):=L^2(X,\nu_X)$  called the 
{\it regular representation} and given by
\begin{equation*}
\label{eqnregrep}
(\la_X(g)\varphi)(x)=\varphi(g^{-1}x)
\;\;\;\mbox{\rm for $g$ in $G$, $\varphi$ in $L^2(X)$ and $x$ in $X$.}
\end{equation*} 

When the class of the measure $\nu_X$ is $G$-invariant, one still has a  
natural unitary representation $\pi$ of $G$ in the 
Hilbert space $L^2(X):=L^2(X,\nu_X)$ also called the 
{\it regular representation}. The formula will involve the 
Radon--Nikodym cocycle $c(g,x)$ which is defined,
for all $g$ in $G$ and $\nu_X$-almost all $x$ in $X$ 
by the equality
\begin{equation}
\label{eqnradnik}
g_*\nu_X=c(g^{-1},x)\,\nu_X\, .
\end{equation} 

The regular representation of $G$ in $L^2(X)$ is then given by  
\begin{equation}
\label{eqnregrec}
(\la_X(g)\varphi)(x)=c(g^{-1},x)^{1/2}\,\varphi(g^{-1}x)
\;\;\;\mbox{\rm for $g$ in $G$, $\varphi$ in $L^2(X)$, $x$ in $X$.}
\end{equation} 

\subsubsection{Induced representations}
Assume now that $X$ is a homogeneous space $G/H$ 
where $H$ is a closed subgroup of $G$.
One can choose a $G$-invariant 
Radon measure on $G/H$
if and only if the modular function of $G$ coincides on $H$ with that of $H$,
\begin{equation*}
\label{eqndegdeh}
\De_G(h) =\De_H(h)
\;\;\;\mbox{\rm for all $h$ in $H$.}
\end{equation*}
In general there  always
exists  a measure $\nu$ on $G/H$ 
whose class is $G$-invariant, and the regular representation 
of $G$ in $L^2(G/H)$ is the induced representation of the trivial representation of $H$
\begin{equation*}
\label{eqnregind}
\la_{G/H}=\Ind_H^G({\bf{1}}).
\end{equation*}
More generally, for any unitary representation $\pi$ of $H$,
one defines the (unitarily) induced representation 
$\Pi:=\Ind_H^G(\pi)$ in the following way.
The projection 
$$
G\longrightarrow G/H
$$ 
is a principal bundle with structure group $H$. 
We fix
a $G$-equivariant Borel measurable trivialization of this principal bundle 
\begin{equation}
\label{eqngsighh}
G\simeq G/H\times H
\end{equation}
which sends relatively compact subsets to relatively compact subsets.
The action of $G$ by left multiplication through this trivialization can be read as 
\begin{equation*}
\label{eqngxhgxs}
g\,(x,h)=(gx,\si(g,x)h)
\;\;\;\;\mbox{\rm 
for all $g\in G$, $x\in G/H$ and $h\in H$,}
\end{equation*}
where $\si\colon G\times G/H\rightarrow H$ 
is a Borel measurable cocycle.

The space of the representation $\Pi$ is the space
$\mc H_{\Pi}:=L^2(G/H;\mc H_{\pi})$ of $\mc H_\pi$-valued $L^2$-functions on $G/H$ and the action of $G$ is given,
for $g$ in $G$, $\psi$ in $\mc H_{\Pi}$, $x$ in $G/H$, by
\begin{equation*}
\label{eqnindrep}
(\Pi(g)\psi)(x)=c(g^{-1},x)^{1/2}\,\pi(\si(g,g^{-1}x))\psi(g^{-1}x),
\end{equation*} 
where $c$ is again the Radon--Nikodym cocycle \eqref{eqnradnik}. 

\subsubsection{Induced actions}
When the closed subgroup $H$ of  $G$ is acting continuously 
on a locally compact space $Z$ one can define the induced action
of $G$ on the fibered space 
$$
G\times_H Z:=  (G\times Z)/H
$$
where the quotient is taken for the right $H$-action 
$(g,z)h=(gh,h^{-1}z)$ and where the $G$-action is given by
$g_0\,(g,z)=(g_0g,z)$, for all $g_0$, $g$ in $G$, $z$ in $Z$ and $h$ in $H$.
Using \eqref{eqngsighh}, one gets a $G$-equivariant Borel measurable trivialization of this fibered space
$$
G\times_H Z\simeq G/H\times Z.
$$ 
Through this identification, the $G$-action
is given by 
$$
g(x,z)=(gx,\si(g,x)z)\, 
\;\;\;\mbox{\rm 
for all $g\in G$, $x\in G/H$ and $z\in Z$.}
$$
When the $H$-action preserves the class of a measure $\nu_Z$ on $Z$, 
the $G$-action preserves the class of the measure $\nu_X:=\nu\otimes\nu_Z$.
In this case the regular representation of $G$ in $L^2(G\times_HZ)$ is unitarily equivalent to the 
induced representation of the regular representation of $H$ in $Z$:
\begin{equation}
\label{eqnrexind} 
L^2(G\times_H Z)\simeq \Ind_H^G(L^2(Z))
\;\;\;\mbox{as unitary representations of $G$.}
\end{equation} 

\subsection{Decay of matrix coefficients}
\label{sectemrep2}

\bq
We now recall the control of the matrix coefficients of tempered 
representations of a semisimple Lie group.
\eq

\subsubsection{Tempered representations}
Let $G$ be a locally compact group and $\pi$ be
a unitary representation  of $G$ in a Hilbert space ${\mathcal H}_\pi $.
All  representations $\pi$ of $G$ will be assumed to be continuous 
{\it{i.e.}} the map $G\ra\mc H_\pi ,\, g\mapsto \pi(g)v$ is continuous
for all $v$ in $\mc H_\pi$.
The notion of tempered representation is due to Harish-Chandra.

\begin{Def}
\label{deftem}
The unitary representation $\pi$ is said to be tempered
or $G$-tempered
if $\pi$ is weakly contained 
in the regular representation $\lambda_G$ of $G$ in $L^2(G)$ {\it{i.e.}} if every matrix coefficient of $\pi$ 
is a uniform limit on every compact subset of $G$  of a sequence of 
sums of matrix coefficients of $\lambda_G$. 
\end{Def}
We refer to \cite[Appendix F]{BeHaVa} for more details on weak containments.

\begin{Rem}
\label{remtem}
The notion of temperedness is stable by passage to a finite index subgroup 
$G'$ of $ G$, {\it{i.e.}}
a unitary representation $\pi$ of  $G$ is tempered if and only if $\pi$
is tempered as a representation of $G'$.
\end{Rem}

This notion is also preserved by induction. 

\begin{Lem}
\label{lemindtem}
Let $G$ be a locally compact group, $H$ be a closed subgroup of $G$
and $\pi$ be a unitary representation of $H$.
If $\pi$ is $H$-tempered then the induced representation
$\Ind_H^G(\pi)$ is $G$-tempered.
\end{Lem}

\begin{proof} 
Since the $H$-representation 
$\pi$ is weakly contained in the regular representation $\la_H$
of $H$,
the $G$-representation $\Ind_H^G(\pi)$ 
is weakly contained in the regular representation $\la_G=\Ind_H^G(\la_H)$, and 
hence  is $G$-tempered.  
\end{proof}

\begin{Rem}
\label{remtemame}
$1)$ When $G$ is amenable, according to the Hulanicki--Reiter theorem 
in \cite[Th.~G.3.2]{BeHaVa}, every unitary representation of $G$ is tempered. 

\noindent 
$2)$ When $G$ is a product of two closed subgroups $G=SZ$ with $Z$ central, {\it a unitary representation $\pi$ of $G$ is 
$G$-tempered if and only if it is $S$-tempered.}
Indeed the regular representation of $G$ in $L^2(G)$ is clearly 
$S$-tempered. Conversely, we want to prove that any unitary representation $\pi$ of $G$ which is $S$-tempered 
is also $G$-tempered.
We can assume that $\pi$ is $G$-irreducible. The action of $Z$ in this representation is given by a unitary character $\chi$ and $\pi$ 
is weakly contained in the representation $\Ind_Z^G\chi$. 
Since $\chi$ is $Z$-tempered, this representation is $G$-tempered. 
\end{Rem}

\subsubsection{Matrix coefficients}
Let  now $G$ be a semisimple Lie group (always implicitly assumed to be real Lie groups with finitely many connected components and whose identity component has finite center). 

\begin{Def}\label{defwl2}
A unitary representation $\pi$ of $G$ is said to be almost $L^2$
if there exists  a dense subset $\mc D\subset{\mathcal H}_\pi $ 
for which the matrix coefficients
 $ g\mapsto \langle  \pi(g)v_1,v_2\rangle $ are in
$L^{2+\varepsilon}(G)$ for all $\varepsilon>0$ and all $v_1$, $v_2$ in $\mc D$. 
\end{Def}

We fix a maximal compact subgroup $K$ of $G$.
Let $\Xi$ be the Harish-Chandra spherical function on $G$ 
(see \cite{CHH}).
By definition, $\Xi$ is  the matrix coefficient of a normalized $K$-invariant vector $v_0$
of the spherical unitary principal representation 
$\pi_0={\Ind}_{P_{{\operatorname{min}}}}^G({\bf 1}_{P_{{\operatorname{min}}}})$ 
where $P_{\operatorname{min}}$ is a minimal parabolic subgroup of $G$.
That is 
\begin{equation}
\label{eqnxighar}
\Xi(g)=\langle\pi_0(g)v_0,v_0\rangle
\;\; 
\mbox{\rm for all $g$ in $G$}. 
\end{equation}
Since $P_{\operatorname{min}}$ is amenable, the representation $\pi_0$ is $G$-tempered. Moreover, the function $\Xi$ belongs to $L^{2+\varepsilon}(G)$ for all $\varepsilon>0$ 
(see \cite[Prop.~7.15]{Kn01}).
We will need the following  much more precise version of
this fact.

\begin{Prop} 
[Cowling, Haagerup and Howe {\cite{CHH}}]
\label{protemal2}
Let $G$ be a connected semisimple Lie 
group with finite center and $\pi$
be a unitary representation  of $G$. The following are equivalent:\\
$(i)$ the representation $\pi$ is tempered,\\
$(ii)$ the representation $\pi$ is almost $L^2$,\\
$(iii)$ for every $K$-finite vectors $v$, $w$ in ${\mathcal H}_\pi $, for every $g$ in $G$,
one has 
$$
|\langle  \pi(g)v,w\rangle |\leq \Xi(g) 
\| v\| \| w\| (\dim \langle  K v\rangle )^\frac12 (\dim \langle  K w\rangle )^\frac12.
$$
\end{Prop}

See \cite[Thms.~1, 2 and Cor.]{CHH}.
See also  \cite{HoTa}, \cite{Nev98} and \cite{Oh1} for other applications
of Proposition \ref{protemal2}.

For  regular representations this proposition becomes:

\begin{Cor}\label{coralpcol}
Let $G$ be a connected semisimple   Lie 
group with finite center
and $X$ a locally compact space endowed with a continuous action of $G$ 
preserving a Radon measure ${\rm vol}$.
The regular representation of $G$ in $L^2(X)$  is tempered if and only if,
for any  compact subset 
$C$ of $X$, the function $g\mapsto {\rm vol} (g\, C\cap C)$ 
belongs to $L^{2+\eps}(G)$
for all $\eps>0$.

In this case, when $C$ is $K$-invariant,
one has
\begin{equation}
\label{eqnvolgcc}
{\rm vol} (g\, C\cap C)\leq {\rm vol}( C)\;\Xi(g) 
\;\;\;\mbox{\rm for all $g$ in $G$}.
\end{equation}
\end{Cor}

Recall that the notation  $g\, C$ denotes the set  $g\, C: =\{gx\in X: x\in C\}$.

\begin{proof}
Note  that  a compact subset $C$ of $X$ is always included
in a $K$-invariant compact subset $C_0$, that
the function ${\bf 1}_{C_0}$ is a $K$-invariant vector in $L^2(X)$ and that 
$$
\langle \la_X(g){\bf 1}_{C_{_0}},{\bf 1}_{C_{_0}}\rangle =
{\rm vol} (g\, C_0\cap C_0)
\;\;\;\mbox{\rm for all $g$ in $G$}.
$$
Note also that  the 
functions  ${\bf 1}_{C}$
span a  dense subspace in $L^2(X)$.
\end{proof}

\subsection{The function $\rho_V$}
\label{secfunrov}

\bq
We now define the functions $\rho_\gs h$ and $\rho_{\gs g/\gs h}$ 
occuring in the temperedness criterion, explain how to compute them  
and emphasize
their geometric meaning.
\eq

When $H$ is a  Lie group we denote by the corresponding gothic letter 
$\g h$ the Lie algebra of $H$. Let $V$ be a real 
finite-dimensional representation of $H$.
For an element $Y$ in $\g h$,
we consider the eigenvalues of $Y$ in $V$
(more precisely in the complexification $V_{\m C}$) and
we denote by $V_+$, $V_0$ and $V_-$ 
the largest vector subspaces of $V$ on which 
the real part of all the eigenvalues of  $Y$
are respectively positive, zero and negative.
One has the decomposition 
$V=V_+\oplus V_0\oplus V_-$.
We define the non-negative functions  
$\rho^+_V$ and $\rho_V$ on $\g h$ by
\begin{eqnarray*}
\label{eqnrhopmv}
\rho^+_V(Y)
&: =& {\rm {\operatorname{Tr}}}(Y|_{V_+}),\\
\rho_V(Y)
&: =& \tfrac12\,\rho^+_V(Y)+\tfrac12\,\rho^+_V(-Y),
\end{eqnarray*}
where ${\rm {\operatorname{Tr}}} $ denotes the trace of a matrix. Note that one has the equality
$
{\rm {\operatorname{Tr}}}(Y|_{V_-})=-\rho^+_V(-Y).
$

By definition, one always has the equality
$\rho_V(-Y) = \rho_V(Y).$ 
Moreover, when the action of $H$ on $V$ is volume preserving
one has the equality
\begin{equation*}
\label{eqnrhovpy}
\rho_V(Y) = \rho^+_V(Y).
\end{equation*}
The function called $\rho_V$  in \cite[Sec. 3.1]{BeKoI} is what we call now $\rho_V^+$. It coincides with our
$\rho_V$  since in \cite{BeKoI} we only need to consider volume preserving actions.

Since this function $\rho_V \colon \g h \to \mathbb{R}_{\ge0}$ 
plays a crucial role in our criterion,
we begin by a few trivial but useful comments, which make  it easy to compute when dealing with examples.
To simplify these comments, we assume that $H$ is an algebraic subgroup of ${\rm GL}(V)$.
Let $\g a=\g a_{\g h}$ be a maximal split abelian Lie subalgebra of $\g h$ 
{\it{i.e.}} the Lie subalgebra of a 
maximal split torus $A$ of $H$.
Any element $Y$ in $\g h$ admits a  unique  Jordan decomposition
$Y=Y_e+Y_h+Y_n$ as a sum of three commuting elements of $\g h$
where $Y_e$ is a semisimple matrix with imaginary eigenvalues,
$Y_h$ is a semisimple matrix with real eigenvalues and 
$Y_n$ is a nilpotent matrix
(see for instance \cite[2.1]{Kos73}). Moreover there exists an element 
$\la_Y$ in $\g a$ which is $H$ conjugate to $Y_h$.
Then one has the equality
\begin{equation*}
\label{eqnrhovly}
\rho_V(Y) = \rho_V(\la_Y)
\;\;\mbox{for all $Y$ in $\g h$.} 
\end{equation*}
This equality tells us that the function $\rho_V$ is completely determined by its restriction to $\g a$.

This function $\rho_V\colon \g a \to \mathbb{R}_{\geq 0}$ is 
continuous and is piecewise linear
{\it{i.e.}} there exist finitely many convex polyhedral cones 
which cover $\g a$ and on which $\rho_V$ is linear.
Indeed, let $P_V$ be the set of weights of $\g a$ in $V$ 
and, for all $\al$ in $P_V$, let $m_\al:=\dim V_{\al}$ 
be the dimension of the corresponding 
weight space. Then one has the equality
\begin{equation}\label{eqnrhovys}
\rho_V(Y) = \tfrac12\sum_{\al\in P_V}m_\al|\al(Y)|
\;\;\;\mbox{for all $Y$ in $\g a$.} 
\end{equation}

For example, when $\g h$ is semisimple and $V = \g h$
 via the adjoint action,
our function $\rho_{\gs h}$ is equal on each  positive Weyl chamber $\g a_+$
of $\g a$ to the sum of the corresponding positive roots {\it{i.e.}} to twice the usual ``$\rho$'' 
linear form.
For other representations $V$, the maximal convex polyhedral cones on which 
$\rho_V$ is linear are most often much smaller than the Weyl chambers. 
Explicit computations of the functions $\rho_V$ will be given in Section \ref{secexalgh}.  
\vs

The geometric meaning of this function $\rho_V$ is given by the following elementary Lemma
as in \cite[Prop. 3.6]{BeKoI}.

\begin{Lem}
\label{lemrovvol}
Let $V=\m R^d$. Let 
$\g a$ be an abelian split Lie subalgebra of ${\rm End}(V)$ and $C$ be a compact neighborhood of $0$ in $V$.
Then there exist constants $m_{_C } >0$ , $M_{_C }  > 0$ such that
\[
m_{_C } e^{-\rho_V(Y)}
\le e^{-\Tr(Y)/2}\,{\rm vol}(e^Y C \cap C)
\le M_{_C }  e^{-\rho_V(Y)}
\;\;\mbox{\rm for all $Y \in \g a$. 
}
\]
\end{Lem}
Such a  factor $e^{-\Tr(Y)/2}$ occurs 
in computing the matrix coefficient of the vector 
${\bf 1}_C$ in the 
regular representation  
$L^2(V)$ when the action  on $V$  does not preserve the  volume.
Here ${\rm vol}$ denotes the volume with respect to
the Lebesgue measure on $V$.
The proof of Lemma \ref{lemrovvol} goes similarly to that of  \cite[Prop. 3.6]{BeKoI} which deals with 
the case where the action is volume preserving.

\subsection{Temperedness criterion for $L^2(G/H)$}
\label{sectemcri}

\bq
We can now state precisely our temperedness criterion.
\eq

Let $G$ be a  semisimple Lie group
and $H$ a closed  subgroup of $G$.
Let $\g g$ and $\g h$ be the Lie algebras of $G$ and $H$.
The temperedness criterion for the regular representation of $G$ 
in $L^2(G/H)$ will involve 
the functions
$\rho_\gs h$ and $\rho_{\gs g/\gs h}$
for the $H$-modules $V=\g h$ and $V=\g g/\g h$.

\begin{Thm}
\label{thmlghtem}
Let $G$ be a connected semisimple Lie group with finite center,
$H$ a closed connected subgroup of $G$.
Then, one has the equivalence~:\\
\centerline{$L^2(G/H)$ is $G$-tempered\;\;\;
$\Longleftrightarrow$\;
$\rho_{\gs h}  \le \rho_{\gs g/\gs h} $.}
\end{Thm}

\begin{Rem} 
The assumption that $G$ and $H$ are connected are not very important. As we shall explain in Corollary \ref{corgh1h2}, Theorem \ref{thmlghtem} is still true when $G$ and $H$ have finitely many connected components as soon as 
the identity component $G_e$ has finite center.
\end{Rem}

\begin{Rem} 
When $H$ is algebraic and 
$\g a$ is a maximal abelian split Lie subalgebra of $\g h$,
Inequality
$\rho_{\gs h}  \le \rho_{\gs g/\gs h} $ holds on $\g h$
if and only if it holds
on~$\g a$.
\end{Rem}

\begin{Rem} 
When $H$ is a minimal parabolic subgroup of $G$ the representation of 
$G$ in $L^2(G/H)$ is tempered because the group $H$ is amenable. 
Our criterion is easy to check in this case since the functions  
$\rho_{\gs h}$ and $\rho_{\gs g/\gs h}$ are equal. 
This example explains why, when $H$ is non-unimodular, our temperedness criterion involves 
the functions $\rho_V$ instead of the functions $\rho^+_V$.
\end{Rem}

\section{Preliminary proofs}
\label{secprepro}

In this section we state a  useful reformulation of Theorem \ref{thmlghtem}
and prove the direct  implication in Theorem \ref{thmlghtem}.

\subsection{The Herz majoration principle}
\label{sechermaj}

\bq
We first explain how to reduce the proof of Theorem \ref{thmlghtem}
to the case where both $G$ and $H$ are algebraic and
how to deal with groups having finitely many connected components.
\eq

\begin{Prop} 
\label{progh1h2}
Let $G$ be a semisimple Lie group with finitely many compo\-nents
such that the identity component $G_e$ has finite center
and
$H'\subset H $ two closed subgroups of $G$.\\
$1)$ If $L^2(G/H)$ is $G$-tempered then $L^2(G/H')$ is $G$-tempered.\\
$2)$ The converse is true when $H'$ is normal in $H$ and $H/H'$ is amenable
(for instance finite,  compact, or abelian).
\end{Prop}

\begin{Lem} 
\label{lemgh1h2}
Let $G$ be a semisimple Lie group with finitely many connected components
such that $G_e$ has finite center,
and $H$ be a closed subgroup of $G$.
If the regular representation in $L^2(G/H)$ is $G$-tempered then the induced representation $\Pi=\Ind_H^G(\pi)$ is also $G$-tempered
for any  unitary representation $\pi$ of $H$.
\end{Lem}

\begin{proof}[Proof of Lemma \ref{lemgh1h2}] This  classical lemma is called ``Herz majoration principle''
(see \cite[Chap.~6]{BeGu}). 
We recall the short argument  since it will be very useful in Proposition \ref{provolgcc}. 
We use freely the notation of Section \ref{secdefres}.
For a function $\ph$ in the space $L^2(G/H,\mc H_\pi)$ 
of the induced representation $\Pi=\Ind_H^G(\pi)$, we denote by   $|\ph|$
the function in the space $L^2(G/H)$ of the regular representation 
$\Pi_0=\Ind_H^G({\bf{1}})$ given by $|\ph|(x):=\|\ph(x)\|$ for   $x$ in $G/H$.
The space $\mc D$ of bounded functions with compact support 
is dense in $L^2(G/H,\mc H_\pi)$. 
For $\ph$ and $\psi$ in $\mc D$, one can compute the matrix coefficients
\begin{eqnarray*}
\langle\Pi(g)\ph,\psi\rangle
&=&\int_{G/H}c(g^{-1},x)^{1/2}\langle\pi(\si(g,g^{-1}x)) \ph(g^{-1} x),\psi(x )\rangle \rmd\nu(x), \\
|\langle\Pi(g)\ph,\psi\rangle|
& \leq&\int_{G/H}c(g^{-1},x)^{1/2} \|\ph (g^{-1} x)\|\,\|\psi(x )\| \rmd\nu(x)\\
& \leq& \langle\Pi_0(g)|\ph|,|\psi|\rangle  .
\end{eqnarray*}
Since $\Pi_0$ is tempered, these matrix coefficients belong to $L^{2+\eps}(G)$ 
for all $\eps>0$. 
Therefore the representation $\Pi$ is almost $L^2$ and hence is  $G$-tempered by Proposition \ref{protemal2}.
\end{proof}

\begin{proof}[Proof of Proposition \ref{progh1h2}] 
$1)$ This follows from Lemma \ref{lemgh1h2} applied to the regular representation $\pi$ of $H$ in $L^2(H/H')$.

$2)$ Since $H/H'$ is amenable, the trivial representation 
of $H$ is weakly contained in the regular representation 
of $H$ in $L^2(H/H')$.
Therefore, inducing to $G$, 
the regular representation of $G$ in $L^2(G/H)$ is 
weakly contained in the 
regular representation of $G$ in $L^2(G/H')$
and hence is $G$-tempered.
\end{proof}

The following corollary tells us that the temperedness of $L^2(G/H)$
depends only on the Lie algebras $\g g$, $\g h$ 
and does not change if we replace $\g h$ by its derived Lie algebra 
$[\g h,\g h]$.

\begin{Cor} 
\label{corgh1h2}
Let $G$ be a semisimple Lie group with finitely many connected components
such that $G_e$ has a finite center $Z_G$
and
$H$ be a closed subgroup with finitely many connected components.
Then the following are equivalent.\\
\centerline{$
(i) \mbox{ $L^2(G/H)$ is $G$-tempered}
\Longleftrightarrow
(ii) \mbox{ $L^2(G_e/H_e)$ is $G_e$-tempered}
\Longleftrightarrow
$}
\centerline{$
(iii) \mbox{ $L^2(G/HZ_G)$ is $G/Z_G$-tempered}
\Longleftrightarrow
(iv)  \mbox{ $L^2(G/[H,H])$ is $G$-tempered.}
$}
\end{Cor}

\begin{proof} 
This follows from Proposition \ref{progh1h2} 
and Remark \ref{remtem} since 
the  quotients $H/H_e$, $HZ_G/H$ and $H/[H,H]$ 
are amenable groups.
\end{proof}

\begin{Rem}
\label{remredalg}
Corollary \ref{corgh1h2} is useful to reduce the proof 
of 
Theorem \ref{thmlghtem} to the case where both $G$ and $H$ are algebraic groups.

Indeed, every semisimple Lie algebra $\g g$ is the Lie algebra 
of an algebraic group: the group ${\rm Aut}(\g g)$. Therefore, 
using $(i)\Leftrightarrow (ii)\Leftrightarrow (iii)$,
we can assume that $G$ is algebraic.

Moreover, by Chevalley's ``th\'eorie des repliques'' in \cite{Che48}, for any closed subgroup $H$ of an algebraic group $G$,
there exists two algebraic subgroups $H_1$ and $H_2$ of $G$
whose Lie algebras satisfy 
$$\g h_1\subset \g h \subset \g h_2
\;\;{\rm and}\;\;\;
\g h_1 = [\g h,\g h]= [\g h_2,\g h_2].
$$
Therefore, 
using $(i)\Leftrightarrow (iv)$, we can assume that $H$ is an algebraic subgroup.
\end{Rem}

\begin{Rem}
\label{remreduni}
Since the  group  $[H,H]$ is   unimodular, Corollary \ref{corgh1h2} is also useful to reduce the proof 
of 
Theorem \ref{thmlghtem} to the case where  $H$ is unimodular.
\end{Rem}

\subsection{A strengthening of the main theorem}
\label{secstrthm}

\bq
Theorem \ref{thmlghtem} will be proven by induction on the dimension of $G$.
This induction process forces us to prove simultaneously an apparently 
stronger theorem which involves $L^2(V)$-valued sections over $G/H$ associated to
a finite-dimensional $H$-module $V$.
\eq

\begin{Thm}
\label{thmlgvtem}
Let $G$ be an algebraic semisimple Lie group,
$H$ an algebraic subgroup of $G$ and 
$V$ a real finite-dimensional algebraic representation of $H$.
Then, one has the equivalence~:\\
\centerline{$\Ind_H^G(L^2(V))$ is $G$-tempered\;\;\;
$\Longleftrightarrow$\;
$\rho_{\gs h}  \le \rho_{\gs g/\gs h} +2\,\rho_V $.}
\end{Thm}

Again, we only need to check this 
inequality on a maximal split abelian Lie subalgebra $\g a$ of $\g h$.
Note also that, by Remark \ref{remredalg}, 
Theorem \ref{thmlghtem} is the special case of 
Theorem \ref{thmlgvtem} where $V=\{ 0\}$.

\subsection{The direct implication}
\label{secdirimp}

\bq
We  first prove the direct implication in Theorems \ref{thmlghtem} and \ref{thmlgvtem}.
\eq

From now on, we will set $\g q:=\g g/\g h$.

\begin{Prop}
\label{prolghtem}
Let $G$ be an algebraic semisimple Lie group,
$H$ an algebraic subgroup of $G$ and $V$ an algebraic representation of $H$.
If the representation $\Pi=\Ind_H^G(L^2(V))$ is $G$-tempered then one has 
$\rho_{\gs h}  \le \rho_{\gs q} +2\rho_V$.
\end{Prop}

\begin{proof}
By \eqref{eqnrexind}  this representation $\Pi$ is also the regular representation of the $G$-space
$X:=G\times_HV$. 
Let $A$ be a maximal split torus of $H$ and $\g a$ be the Lie algebra of $A$.
We choose an $A$-invariant decomposition 
$\g g=\g h\oplus \g q_0$ and small closed balls $B_0\subset \g q_0$ and
$B_V\subset V$ centered at $0$. We can see $B_V$ as a subset of $X$ and the map
\begin{eqnarray}
\label{eqnbobvgv}
B_0\times B_V \longrightarrow G\times_H V, 
\qquad
(u,v)\mapsto \exp (u)v 
\nonumber
\end{eqnarray}
is a homeomorphism onto its image $C$. 
Since $\Pi$ is tempered one has a bound
as in   \eqref{eqnvolgcc}
\begin{equation}
\label{eqnpigxig}
\langle\Pi(g )1_C,1_C\rangle\leq M_C\; \Xi(g)
\;\;\;\mbox{\rm for all $g$ in $G$}.
\end{equation}
We will exploit this bound for elements $g=e^Y$ with $Y$ in $\g a$.
In our coordinate system \eqref{eqnbobvgv} we can choose 
the measure $\nu_X$ to coincide with the Lebesgue measure on 
$\g q_0\oplus V$. Taking into account the Radon--Nykodim derivative and the $A$-invariance of $\g q_0$,
one computes
\begin{eqnarray*}
\label{eqnpigvol}
\langle\Pi(e^Y)1_C,1_C\rangle
\!\!&\geq &\!\! 
e^{-{Tr}_{\gs q_0}(Y)/2} e^{-{Tr}_{V}(Y)/2}\;
{\rm vol}_{\gs q_0}(e^YB_0\cap B_0)\;{\rm vol}_V(e^YB_V\cap B_V),
\end{eqnarray*}
and therefore, using Lemma \ref{lemrovvol}, one deduces
\begin{eqnarray}
\label{eqnpigrho}
\langle\Pi(e^Y )1_C,1_C\rangle
&\geq &m_{_C}\,e^{-\rho_{\gs q}(Y) } e^{-\rho_V(Y)} 
\;\;\;\mbox{\rm for all $Y$ in $\g a$.} 
\end{eqnarray}
Combining  \eqref{eqnpigxig} and \eqref{eqnpigrho} 
with known bounds for the 
spherical function $\Xi$ as in \cite[Prop 7.15]{Kn01},
one gets, for suitable positive constants $d$, $C$,
\begin{equation*}
\label{eqnroxiro}
\frac{m_C}{M_C}e^{-\rho_{\gs q}(Y)  -\rho_V(Y) }\leq \Xi(e^Y)\leq M_0\;(1+\|Y\|)^{d}e^{-\rho_{\gs g}(Y)/2 }  
\;\;\;\mbox{\rm for all $Y$ in $\g a$.} 
\end{equation*}
Therefore one has $\rho_{\gs g}\leq 2\,\rho_{\gs q}+2\,\rho_V$, 
and hence  $\rho_{\gs h}\leq  \rho_{\gs q}+2\,\rho_V$ as required.
\end{proof}

\subsection{Equivalence of the main theorems}
\label{secequthm}
\bq
We have already noticed that 
Theorem \ref{thmlghtem} is a special case of Theorem \ref{thmlgvtem}.
We explain now why Theorem \ref{thmlgvtem} is a consequence of 
Theorem \ref{thmlghtem}.
\eq

\begin{Prop}
\label{proequthm}
Let $G$ be an algebraic semisimple Lie group.
If the conclusion of Theorem \ref{thmlghtem} is true for all  algebraic subgroups $H$ of $G$, then 
the conclusion of Theorem \ref{thmlgvtem} is also true for all  algebraic subgroups $H$ of $G$.
\end{Prop}

The proof relies on the following lemma.

\begin{Lem}
\label{lemtratri}
Let $H$ be a Lie group and $V$ a finite-dimensional representation of $H$.
Let $v\in V$ be a point  whose orbit $Hv$ has maximal dimension and $H_v$ be the stabilizer of $v$ in $H$. 
Then the action of $H_v$ on  $V/\g h v$ is trivial.  
\end{Lem}

\begin{proof}[Proof of Lemma \ref{lemtratri}]
Assume by contradiction that there exist $Y$ in $\g h$ 
and $w$ in $V$ such that the vector 
$Yw$ does not belong to $\g h v$.

Choose a complementary subspace $\g m$ of $\g h_v$ in $\g h$ so that $\g h= \g h_v\oplus \g m$.
Choose also a point $v_\eps= v+\eps w$ 
near $v$.  For $\eps$ small,
the tangent space $\g h\, v_\eps$ to the orbit
$H\, v_\eps$ contains both the subspace $\g m\, v_\eps$
which is near $\g m \,v=\g h\, v$
and the vector $\eps^{-1}Yv_\eps=Y w$.
Therefore, for $\eps$ small, one has the inequality
$\dim \g h v_\eps>\dim \g h v$ which gives us a contradiction.
\end{proof}

\begin{proof}[Proof of Proposition \ref{proequthm}] We assume that $\rho_{\gs h}\leq  \rho_{\gs q}+2\,\rho_V$ and we want to prove, using Theorem \ref{thmlghtem}, that the regular representation of $G$
in $L^2(G\times_H V)$ is tempered.
Since the action is algebraic, there exists a Borel measurable subset 
$T\subset V$ which meets each of these $H$-orbits in exactly one point.
Let $\nu_V$ be a probability measure on $V$ with positive density and $\nu_T$ be the probability measure on $T\simeq H\backslash V$ given as the image of $\nu_V$.
One has a direct integral decomposition 
of the regular representation
\begin{equation*}
\label{eqnregint}
L^2(G\times_H V) =\int_T^{\oplus}L^2(G/H_v)\rmd\nu_T(v)
\end{equation*}
where $H_v$ is the stabilizer of $v$ in $H$.  
Since the direct integral of tempered represen\-tations is tempered,
we only need to prove that, for $\nu_T$-almost all $v$ in $T$,
\begin{equation}
\label{eqnghvtem}
L^2(G/H_v) 
\;\;\mbox{is $G$-tempered.}
\end{equation}
Our assumption implies that 
\begin{equation}
\label{eqnrhohqv}
\rho_{\gs h}(Y)\leq  \rho_{\gs q}(Y)+2\,\rho_V(Y)
\;\;\mbox{for all $Y$ in $\g h_v$}
\end{equation}
For $\nu_T$-almost all $v$ in $T$, the orbit $Hv$ has maximal dimension,
hence, by Lemma \ref{lemtratri},
the action of $\g h$ on the quotient $V/\g h v$ is trivial,
and therefore one has the equality 
\begin{equation}
\label{eqnrhovhv}
\rho_{V}(Y)=  \rho_{\gs h}(Y)-\rho_{\gs h_v}(Y)
\;\;\mbox{for all $Y$ in $\g h_v$}
\end{equation}
Combining \eqref{eqnrhohqv} and \eqref{eqnrhovhv}, one gets,
for $\nu_T$-almost all $v$ in $T$,
\begin{equation*}
\label{eqnrhohvg}
2\,\rho_{\gs h_v}(Y)\leq
\rho_{\gs q}(Y)+\rho_{\gs h}(Y)=\rho_{\gs g}(Y)
\;\;\mbox{for all $Y$ in $\g h_v$}
\end{equation*}
which can be rewritten as 
the temperedness criterion 
$
\rho_{\gs h_v}(Y)\leq
\rho_{\gs q_v}(Y)
$ 
for 
$L^2(G/H_v)$ in Theorem \ref{thmlghtem} and hence proves \eqref{eqnghvtem}.
\end{proof}

\section{Using parabolic subgroups}
\label{secparsub}

The aim of this section is to prove the converse implication in 
Theorem \ref{thmlghtem}. As we have seen in Remarks
\ref{remredalg} and \ref{remreduni}, 
we can assume that  $G$ is a Zariski connected algebraic group and that 
$H$ is a Zariski connected algebraic subgroup such that $\g h=[\g h,\g h]$.

The  proof relies on the presence of two nice intermediate 
subgroups 
$$
H\subset F \subset P\subset G.
$$

\subsection{The intermediate subgroups}
\label{secintpar}

\bq
We first explain the construction of these intermediate subgroups
$F$ and $P$.
\eq

Let $G$ be an algebraic semisimple Lie group
and $H$ a Zariski connected  algebraic subgroup of $G$
such that $\g h=[\g h,\g h]$.

\begin{LemDef}
\label{defparsub}
We fix a parabolic subgroup $P$ of $G$ of minimal dimension
that contains $H$ and denote by $U$ the unipotent radical of $P$. 
There exists a reductive subgroup $L\subset P$ such that 
$P=LU$ and $H=(L\cap H)(U\cap H)$. Moreover the group
$S:=L\cap H$ is semisimple and the group $V:=U\cap H$ 
is the unipotent radical of $H$.
We denote by $F$ the group  $F=SU$.
\end{LemDef}

\begin{proof}
The group $V:=U\cap H$ is a unipotent normal subgroup of $H$. 
The quotient $S':=H/V$ is a Zariski connected  subgroup of
the reductive group $P/U$ which is not contained in any proper
parabolic subgroup of $P/U$. 
Therefore, by \cite[Sec.~VIII.10]{Bou7a9}
this group $S'$ is reductive. 
Since $\g h=[\g h,\g h]$, this group $S'$ is semisimple
and there exists a semisimple subgroup $S\subset H$ such that 
$H=SV$. Since $S$ is semisimple, the group $V$ is the unipotent radical of $H$.
Since maximal reductive subgroups $L$ of $P$ are $U$-conjugate,
one can choose $L$ containing $S$ and therefore one has $S=L\cap H$.
\end{proof}

The following two lemmas will be useful 
in our induction process.

\begin{Lem}
\label{lemparsub}
With the notation of Definition \ref{defparsub},
the following two functions on $\g s$ are equal~:
\begin{equation}
\label{eqnparsub}
\rho_{\gs g/\gs h}-\rho_{\gs h}=
\rho_{\gs l/\gs s}+2\, \rho_{\gs u/\gs v}
-\rho_{\gs s}. 
\end{equation}
\end{Lem}

\begin{proof} Since $\rho_{\gs g/\gs p}=\rho_{\gs u}$,
one has the equalities of functions on $\g s$,\\
\mbox{ }\hfill
$
\rho_{\gs g/\gs h}=
\rho_{\gs u}+\rho_{\gs l/\gs s}+ \rho_{\gs u/\gs v}
\;\;\; {\rm and}\; \;\;\;
\rho_{\gs h}=
\rho_{\gs s}+ \rho_{\gs v}.
$\hfill
\end{proof}

\begin{Lem}
\label{lemparsub2}
Let $P=LU$ be a real algebraic group which is
a semidirect product of a reductive subgroup $L$ and its unipotent radical $U$. Let $\pi_0$ be a unitary representation
of $P$ which is $L$-tempered and trivial on $U$.
Then the representation $\pi_0$ is also $P$-tempered.
\end{Lem}

\begin{proof}
The weak containment $\pi_0 \prec L^2(L)$ as unitary representations of $L$
implies 
the weak containment $\pi_0 \prec L^2(P/U)$ as unitary representations of $P$
because $U$ acts trivially on the both sides.
Since $U$ is amenable, the
trivial representation of $U$ is $U$-tempered,
therefore  by Lemma \ref{lemindtem} the
regular representation of $P$ in $L^2(P/U)$ is 
$P$-tempered, and $\pi_0$ is also $P$-tempered.
\end{proof}

\subsection{Bounding volume of compact sets}
\label{secvolcom}
\bq
The proof of Theorem \ref{thmlghtem} relies on a control of the volume of the intersection of translates of compact sets in $X=G/H$.
We first explain how to bound such  volumes in $Z=F/H$.    
This bound is  quite general. 
\eq

\begin{Prop}
\label{prodom}
Let $F=SU$ be a real algebraic  group which is a semidirect product of a reductive subgroup $S$ and 
its unipotent radical $U$. 
Let $H=SV$ be an algebraic subgroup of $F$ containing
$S$  where $V=U\cap H$. 
Let $Z$ be the $F$-space $Z=F/H=U/V$ endowed with a $U$-invariant Radon measure.
Then for every compact subset $D\subset Z$, there exists a compact subset 
$D_0\subset Z$ such that for all $s\in S$ and $u\in U$, one has
\begin{equation}
\label{eqnvolsus}
{\rm vol}(su D\cap D)\leq {\rm vol}(sD_0\cap D_0).
\end{equation}
\end{Prop}

Here is the reformulation 
of Proposition \ref{prodom} that  will be used later on.
\begin{Def}
\label{defzousuh}
Let $Z_0$ be the  same space $Z_0= U/V$ as $Z$ 
but endowed with another $F$-action 
where $U$ acts trivially
and where $S$ acts by conjugation. 
\end{Def}

\begin{Cor}
\label{cordom}
Same notation as in Proposition \ref{prodom}.
Then for every compact subset $D\subset Z$, there exists a compact subset 
$D_0\subset Z_0$ such that for every $f\in F$, one has
\begin{equation}
\label{eqnvolfdd}
{\rm vol}(f D\cap D)\leq {\rm vol}(fD_0\cap D_0).
\end{equation}
\end{Cor}
The proof of Proposition \ref{prodom} is by induction on the dimension of $Z$. It relies only on geometric arguments and 
uses no representation theory.
\vs

Before studying the proof of Proposition \ref{prodom} 
the reader could as an exercise focus on the following very simple example where 
$Z=\m R^2$ is the affine $2$-plane and 
$F$ is the group of affine bijections
$\left(\begin{matrix}a&r\\ 0&b\end{matrix}\right)
\left(\begin{matrix}t\\ s\end{matrix}\right)\; $
that preserve
the horizontal foliation. In this case $S$ is the 
$2$-dimensional group 
of diagonal matrices and $U$ is the $3$-dimensional
Heisenberg group.
The proof for this example relies on the same ideas
while being very concrete.

\begin{proof}[Proof of Proposition \ref{prodom}]
{\bf First case: }$S$ is a split torus. 

We denote by $C$ the center of $U$ and $C_V=V\cap C$.
Let $W$ be the closed subgroup $W:=VC\subset U$. The projection
$$
Z=U/V\longrightarrow Z':=U/W
$$
is a principal bundle of group $C_W:=C/C_V=W/V$.
According to Lemma \ref{lemdom} below, there exists a continuous 
trivialization of this principal bundle
\begin{equation}
\label{eqnuvuwzz}
Z\simeq Z'\times C_W
\end{equation}
such that the action of $U$ and $S$ through this trivialization can be read as 
\begin{eqnarray}
\label{eqnuyzcuy}
su\,(z',c)&=& (su\, z'\, ,\, s\,c+s\,c_0(u,z'))
\end{eqnarray}
for all $s\in S$, $u\in U$, $z'\in Z'$, 
$c\in C_W$, 
where $c_0$ is a continuous cocycle $c_0:U\times Z'\rightarrow C_W$. We fix three compatible invariant measures 
${\rm vol}_Z$, ${\rm vol}_{Z'}$ and ${\rm vol}$ on $Z$, $Z'$, and $C_W$.

We start with a compact set $D\subset Z$.
Through the trivialization \eqref{eqnuvuwzz}, 
this set $D$ is included in a product 
of two compact sets $D'\subset Z'$ and $B\subset C_W$
\begin{equation*}
D\subset D'\times B\, ,
\end{equation*}
where $B$ is a symmetric convex set in the group $C_W$
seen as a real vector space. 
By the induction hypothesis,
there exists a compact set $D'_0\subset Z'$
which satisfies 
the bound \eqref{eqnvolsus} for $D'$, {\it{i.e.}} such that
\begin{equation}
\label{eqnvolfd2}
{\rm vol}_{Z'}(su D'\cap D')\leq {\rm vol}_{Z'}(sD'_0\cap D'_0)
\;\;\;\mbox{for all $s\in S$, $u\in U$}.
\end{equation}
We  compute using \eqref{eqnuyzcuy} 
and Lemma \ref{lemconset} below, 
for all $s\in S$ and $u\in U$,

\begin{eqnarray*}
{\rm vol}_Z(su D\cap D)
&\leq&
\int_{suD'\cap D'}{\rm vol}((sB+s\,c_0(u,(su)^{-1}z'))\cap B)\rmd z'\\
&\leq &
\int_{suD'\cap D'}{\rm vol}(sB\cap B)\rmd z'
\end{eqnarray*} 
where $\rmd z'$ also denotes the $U$-invariant measure on $Z'$.
Hence, using  \eqref{eqnvolfd2}, we go on
\begin{eqnarray*}
{\rm vol}_Z(su D\cap D)
&\leq&{\rm vol}_{Z'}(suD'\cap D')\;{\rm vol}(sB\cap B)\\
&\leq &
{\rm vol}_{Z'}(sD_0'\cap D_0')\; {\rm vol}(sB\cap B)\\
&= &
{\rm vol}_{Z}(sD_0\cap D_0),
\end{eqnarray*} 
where $D_0$ is the compact subset of $Z$
given by  $D_0:=D'_0\times B$.
\vs

{\bf Second case: }$S$ is a reductive group.

This general case will be deduced from the first case.
Indeed any reductive group admits a Cartan decomposition
$S=K_{_S}A_{_S}K_{_S}$ where $K_{_S}$ is a maximal compact subgroup 
of $S$ and where $A_{_S}$ is a maximal split torus of $S$.
We start with a compact set $D$ of $Z$.
According to the first case, 
there exists a $K_{_S}$-invariant compact set $D_0\subset Z$
such that, for all $a\in A_{_S}$ and  $u\in U$, one has
\begin{equation*}
\label{eqnvolfd3}
{\rm vol}(au K_{_S}D\cap K_{_S}D)\leq {\rm vol}(aD_0\cap D_0).
\end{equation*}
Therefore, for all $s$ in $S$ and $u$ in $U$, writing
$s=k_1ak_2$ with 
$k_1$, $k_2$ in $K_{_S}$ and $a$ in $A_{_S}$, one has 
\begin{eqnarray*}
\label{eqnvolfd4}
{\rm vol}(su D\cap D)
&\leq& 
{\rm vol}(a(k_2uk_2^{-1})k_2D\cap k_1^{-1}D)\\
&\leq &
{\rm vol}(aD_0\cap D_0)\\
&= &
{\rm vol}(sD_0\cap D_0),
\end{eqnarray*}
as required.
\end{proof}
In the proof of Proposition \ref{prodom}, we have used the following
two lemmas.

\begin{Lem}
\label{lemdom}
Let $U$ be a unipotent group, $V\subset U$ a unipotent subgroup, $C$ be the center of $U$, $W:=VC$ and  
$C_V:=C\cap V$.
Let $S\subset {\rm Aut}(U)$ be a split torus 
which preserves $V$.
Then 
there exists a continuous 
trivialization of the $U$-equivariant principal bundle $U/V\rightarrow U/W$ with structure group $C/C_V$
\begin{equation*}
\label{eqnuvuwzz2}
U/V\simeq U/W\times C/C_V
\end{equation*}
such that the action of $U$ and $S$ through this trivialization can be read as 
\begin{eqnarray*}
\label{eqnuyzcuy2}
su\,(y,c)&=& (su\, y\, ,\, s\,c+s\,c_0(u,y))
\end{eqnarray*}
for all $u\in U$, $s\in S$, $y\in U/W$, 
$c\in C/C_V$, 
where $c_0$ is a continuous cocycle $c_0\colon U\times U/W\rightarrow C/C_V$.
\end{Lem}

\begin{proof}[Proof of Lemma \ref{lemdom}]
These claims are a variation of a classical result of 
Chevalley--Rosenlicht (see for instance \cite[Thm.~3.1.4]{CoGr}).
The proof relies on the existence of ``an adapted basis in a nilpotent
Lie algebra''. 
Here is a sketch of proof of these claims.

As usual, let $\g u$, $\g v$, $\g c$ and $\g w$
be the Lie algebras of the groups $U$, $V$, $C$ and $W$.
Let $I$ be the ordered set  $I=\{1,\ldots,n\}$, where $n=\dim \g u$. We fix a basis $(e_i)_{i\in I}$ of
$\g u$, such that\\
- for every $i\geq 1$, the 
vector space spanned by the $e_j$ for $j\geq i$ is an ideal;\\
- for every $i\geq 1$, the line $\m R e_i$ is invariant by $S$;\\
- there exists a subset $I_V\subset I$ such that 
$\g v$ is spanned by  $e_i$ for $i\in I_V$; \\
- there exists a subset $I_C\subset I$ such that 
$\g c$ is spanned by  $e_i$ for $i\in I_C$; \\
- the Lie algebra
$\g w$ is spanned by  $e_i$ for $i\in I_W:= I_C\cup I_V$. 

Then, the map
\begin{eqnarray*}
\Psi\;\colon\;\;\;\m R^I
\longrightarrow
U
\quad
(t_i)_{i\in I}\mapsto 
\textstyle\prod _{i\in I}\exp(t_ie_i),
\end{eqnarray*} 
where the product is performed using the order on $I$,
is a diffeomorphism
and one has 
$$
\Psi(\m R^{I_V})=V
\;\; ,
\;\;
\Psi(\m R^{I_C})=C
\;\;{\rm and}\;\; 
\Psi(\m R^{I_W})=W.
$$
Setting $J_V:= I\smallsetminus I_V$
and  $J_W:= I\smallsetminus I_W$,
the map 
\begin{eqnarray*}
\Psi_V\;\colon\;\;\;\m R^{J_V}
\longrightarrow
U/V, 
\quad
(t_i)_{i\in J_V} \mapsto 
\textstyle\prod _{i\in J_V}\exp(t_ie_i)V
\end{eqnarray*} 
is also a diffeomorphism and the restriction of this map to the subset
$\m R^{J_W}$ gives an $S$-equivariant section 
of the bundle $U/V\rightarrow U/W$.
\end{proof}

Here is the second basic lemma  used in the proof
of Proposition \ref{prodom}.

\begin{Lem}
\label{lemconset}
Let $B$, $B'$ be two symmetric convex sets of $\m R^d$,
then one has
$$
{\rm vol}((B+v)\cap B'))\leq {\rm vol} (B\cap B')
\;\;\;\mbox{\rm for all $v\in \m R^d$}.
$$
\end{Lem}
\begin{proof}
By the Brunn--Minkowski inequality (see \cite[Sec.~11]{BoFe}),
the map 
$v\mapsto {\rm vol}((B+v)\cap B')^{1/d}$ is concave on the convex set $B'-B$ and hence achieves its maximum value 
at $v=0$.
\end{proof}

\subsection{Matrix coefficients of induced representations}
\label{seccoeind}
\bq
We now explain how to control 
the volume of the intersection of 
translates of compact sets in the $G$-space $X=G/H$
with those in $X_0:=G\times_F Z_0$.    
\eq

\begin{Prop}
\label{provolgcc}
Let $G$ be an algebraic semisimple Lie group
and $H$ a Zariski connected  algebraic subgroup
such that $\g h=[\g h,\g h]$.
Let $P=LU$, $F=SU$ and $H=SV$ 
be the groups introduced in Definition $\ref{defparsub}$.
Let $Z_0$ be the $F$-space  introduced in Definition 
$\ref{defzousuh}$ and $X_0$ the $G$-space $X_0:=G\times_FZ_0$. 
Then, for every compact subset $C\subset G/H$, there exists a compact subset 
$C_0\subset X_0$ such that 
\begin{equation}
\label{eqnvolcc02}
{\rm vol}(g\, C\cap C) \leq {\rm vol}(g\, C_0\cap C_0)
\;\;\;\; \mbox{\rm for all $g$ in $G$.}
\end{equation}
\end{Prop}

In Proposition \ref{provolgcc} the assumption $\g h=[\g h,\g h]$ can be removed but the conclusion \eqref{eqnvolcc02}
becomes slightly 
more technical when there is no $G$-invariant measure on $G/H$. Indeed, when $\g h\neq [\g h,\g h]$, one has to replace 
the bound \eqref{eqnvolcc02} by a  
bound of $K$-finite matrix coefficients of the induced representation $\Pi=\Ind_F^G(L^2(F/H))$ thanks to $K$-finite matrix  coefficients of the induced representation $\Pi_0=\Ind_F^G(L^2(Z_0))$.

\begin{proof}[Proof of Proposition \ref{provolgcc}]
The projection 
$$
G\rightarrow X':=G/F
$$
is a $G$-equivariant principal bundle with structure group $F$. 
As in Section \ref{secregind}, we fix
a Borel measurable trivialization of this principal bundle 
\begin{equation}
\label{eqngsigh2}
G\simeq X'\times F
\end{equation}
which sends relatively compact subsets to relatively compact subsets.
The action of $G$ by left multiplication through this trivialization can be read as 
\begin{equation*}
\label{eqngxhgx2}
g\,(x',f)=(gx',\si_F(g,x')f)
\;\;\;\;\mbox{\rm 
for all $g\in G$, $x'\in X'$ and $f\in F$,}
\end{equation*}
where $\si_F\colon G\times X'\rightarrow F$ 
is a Borel  measurable cocycle. 
This trivialization \eqref{eqngsigh2}
induces a trivialization of the associated bundles
\begin{eqnarray*}
\label{eqngsigh3}
X= G\times _F Z    
&\simeq&  X' \times Z\, , \\    
X_0 = G \times_F Z_0  
&\simeq &                 
X' \times Z_0\, .
\end{eqnarray*}
We start with a compact set $C$ of $X$.
Through the first trivialization, 
this compact set is included in a product 
of two compact sets $C'\subset X'$ and $D\subset Z$
\begin{equation}
\label{eqncsumxd}
C\subset C'\times D\; .
\end{equation}
We denote by $D_0\subset Z_0$ the compact set given
by Corollary \ref{cordom} 
and we compute using \eqref{eqnvolfdd}, 
for $g$ in $G$,
\begin{eqnarray*}
{\rm vol}_X(g\, C\cap C)
&\leq&
\int_{gC'\cap C'}{\rm vol}_Z(\si_F(g,g^{-1} x')D\cap D)\rmd x'\\
&\leq &
\int_{gC'\cap C'}{\rm vol}_{Z_0}(\si_F(g,g^{-1} x')D_0\cap D_0)\rmd x'\\
&\leq &
{\rm vol}_{X_0}(g\, C_0\cap C_0),
\end{eqnarray*} 
where $\rmd x'$ is a $G$-invariant measure on $X'$
and $C_0$ is a compact subset of $X_0\simeq X'\times Z_0$
which contains $C'\times D_0$.
\end{proof}

\subsection{Proof of the temperedness criterion}
\label{secsemgro}

\bq
We conclude the proof of Theorem \ref{thmlghtem}.
\eq

\begin{proof}[Proof of the converse implication in Theorem \ref{thmlghtem}] We prove it by induction on the dimension of $G$.
By Remarks
\ref{remredalg} and \ref{remreduni}, 
we can assume that  $G$ is a Zariski connected semisimple algebraic group and that 
$H$ is a Zariski connected algebraic subgroup such that $\g h=[\g h,\g h]$.
Let 
$$
H=SV\subset F=SU\subset P=LU \subset G
$$
be the groups introduced in Definition $\ref{defparsub}$.
Let $Z_0=U/V$ be the $F$-space introduced in Definition 
$\ref{defzousuh}$ and $X_0$ be the $G$-space $X_0=G\times_FZ_0$. 

When $P$ is equal to $G$, the group $H$ is semisimple and we apply \cite[Thm 3.1]{BeKoI}. We now assume that $P$ is a proper parabolic subgroup of $G$. 

By assumption one has $\rho_\gs h\leq \rho_{\gs g/\gs h}$ on $\g h$. 
Therefore, by Lemma \ref{lemparsub}, 
one has the inequality on $\g s$
\begin{equation}
\label{eqnparsub2}
\rho_{\gs s}\leq \rho_{\gs l/\gs s}+2\, \rho_{\gs u/\gs v}. 
\end{equation}

We introduce the regular representation $\pi_0$ of $P$ in 
$L^2(P\times_FZ_0)$
 which is unitarily equivalent 
 to $\Ind_F^P(L^2({\mathfrak{u}}/{\mathfrak{v}}))$
 by the isomorphism \eqref{eqnrexind}.
As a representation of $L$, one has 
$$
\pi_0|_L=\Ind_S^L(L^2(\g u/\g v)).
$$
Using our induction hypothesis
on the dimension of G
to the derived subgroup of $L$,
Proposition \ref{proequthm} and Remark \ref{remtemame} 
tell us 
that the representation $\pi_0$ is $L$-tempered
 by \eqref{eqnrexind}.
Therefore, by Lemma \ref{lemparsub2},  
the representation $\pi_0$ is $P$-tempered. 
The regular representation  $\Pi_0$ of $G$ in $L^2(G\times_FZ_0)$ 
is unitarily equivalent to $\Pi_0 =\Ind_P^G(\pi_0)$ 
because
$
\Pi_0 \simeq \Ind_P^G (\Ind_F^P (L^2(\g u/\g v))).
$
Now, Lemma \ref{lemindtem} implies that 
this representation $\Pi_0$ is $G$-tempered.
Therefore, by Corollary \ref{coralpcol},
for any $K$-invariant
compact subset $C_0$ of $G\times_FZ_0$, 
one has a bound:
\begin{equation*}
{\rm vol}(g\, C_0\cap C_0) \leq {\rm vol}( C_0)\; \Xi(g) 
\;\;\;\; \mbox{\rm for all $g$ in $G$.}
\end{equation*}
Hence, by Proposition \ref{provolgcc}, 
for any 
compact subset $C$ of $G/H$, 
one also has a bound
\begin{equation*}
{\rm vol}(g\, C\cap C) \leq M_C\; \Xi(g)
\;\;\;\; \mbox{\rm for all $g$ in $G$.}
\end{equation*}
Again by Corollary \ref{coralpcol}, this tells us 
that the representation of $G$ in $L^2(G/H)$ 
is $G$-tempered.
\end{proof}

\section{Examples}
\label{secexalgh}

The criterion given in Theorem \ref{thmlghtem}  allows to easily detect for a given homogeneous space $G/H$ whether 
the unitary representation of a semisimple Lie group $G$ in $L^2(G/H)$ is tempered or not. 
We collect in this chapter a few examples,
omitting the details of the computational verifications.

\subsection{Examples of tempered homogeneous spaces}
\label{secredhom}

\bq
We first recall a few examples  
extracted from \cite{BeKoI} where $H$ is reductive.
\eq

\begin{Exa}
\label{exagslopq}
$L^2(SL(p+q,\mathbb{R})/SO(p,q))$ is always tempered.\\
$L^2(SL(2m,\mathbb{R})/Sp(m,\m R))$ is never tempered.\\
$L^2(SL(m+n,\mathbb{C})/SL(m,\mathbb{C})\times SL(n,\mathbb{C}))$ 
is tempered iff $|m-n|\leq 1$.\\
$L^2(SO(m+n,\mathbb{C})/SO(m,\m C))\times SO(n,\m C))$ is tempered iff
$|m-n|\leq 2$.\\
$L^2(Sp(m+n,\mathbb{C})/Sp(m,\mathbb{C})\times Sp(n,\mathbb{C}))$ is tempered iff $m=n$.
\end{Exa} 

\begin{Exa}
\label{exaslprod}
Let $n=n_1+\cdots + n_r$ with $n_1\geq\cdots\geq n_r\geq 1$, $r\geq 2$.\\
$L^2(SL(n,\m R)/\prod SL(n_i,\m R))$ is tempered iff 
$2n_1\leq n+1$.\\
$L^2(Sp(n,\m R)/\prod Sp(n_i,\m R))$ is tempered iff
$2n_1\leq n$.\\
Let $p=p_1+p_2$, $q=q_1+q_2$ with $p_1$, $p_2$, $q_1$, $q_2$ $\geq 1$.\\
$L^2(SO(p,q)/ SO(p_1,q_1)\times SO(p_2,q_2))$ is tempered iff
$|p_1+q_1-p_2-q_2|\leq 2$.
\end{Exa}

\begin{Exa}
\label{exagxg}
Let $G $ be an algebraic semisimple Lie group and 
$K $ a maximal compact subgroup.
$L^2(G_{\m C}/K_{\m C})$ is $G_{\m C}$-tempered iff $G $ is quasisplit.
\end{Exa}

\begin{Rem}
A  way to justify this last example is to
notice that our criterion 
$2\, \rho_{\g k_\m C}\leq \rho_{\g g_\m C}$
means that 
{\it the trivial $K$-type is a small $K$-type of $G$}
in the terminology of Vogan's paper 
\cite[Def. 6.1]{Vog79}, see also in Knapp's book
\cite[Chap.~XV]{Kn01}, and to use 
the following equivalences due to Vogan in the same paper
\cite[Thm.~6.4]{Vog79}:
$$
\mbox{
$G$ has a small $K$-type 
$\Longleftrightarrow$
the trivial $K$-type is small
$\Longleftrightarrow$
$G$ is quasisplit.}
$$
\end{Rem}

Here is a delicate example for semisimple symmetric spaces.
\begin{Exa}
\label{exSp} Let $G/H:=Sp(2,1)/Sp(1) \times Sp(1,1)$.
The Plancherel formmula \cite{Del98, Osh88} tells that 
both the continuous part and  a ``generic portion" of the discrete part
of $L^2(G/H)$ are tempered, however, 
 our criterion \eqref{eqnlghtem} tells that $L^2(G/H)$ is non-tempered
because $\rho_{\mathfrak h}(Y) = \frac{3}{2} \rho_{\mathfrak q}(Y) >\rho_{\mathfrak q}(Y)$
 if $Y$ is a nonzero hyperbolic element of $\mathfrak h$.
In fact, the discrete part of $L^2(G/H)$ 
consist of Harish-Chandra's discrete series representations, say $\pi_n$ ($n=1,2,...$),
and two more non-vanishing representations $\pi_{0}$ and $\pi_{-1}$
in the coherent family, 
 where $\pi_0$ is still tempered but $\pi_{-1}$ is non-tempered (\cite[Thm.~1]{kob92}).
\end{Exa}

Here is another direct application of our criterion \eqref{eqnlghtem} where $H$ is not anymore assumed to be reductive.

\begin{Cor} 
\label{corgh}
Let $G$ be an algebraic  semisimple Lie group,
and
$H$ an algebraic  subgroup.\\
$1)$ If the representation of $G_{\mathbb{C}}$ in
$L^2(G_{\mathbb{C}}/H_{\mathbb{C}})$ is tempered,
then the representation of $G$ in $L^2(G/H)$ is tempered.\\
$2)$ The converse is true under if $H$ contains a maximal torus 
which is split.
\end{Cor}


\subsection{Subgroups of $SL(n,\m R)$}
\label{secnonred}
\bq
We now explain how to check our criterion \eqref{eqnlghtem}
on a very concrete example.
\eq

In Table 1, we specify our criterion \eqref{eqnlghtem} when
$
G= {\rm SL}(\m R^p\oplus\m R^q) 
$
and  $H$ is a  subgroup of $G$  
normalized by the 
group ${\rm SL}(\m R^p)\times {\rm SL}(\m R^q) $.
\begin{figure}[htb]
\begin{center}
\begin{tabular}{|c|c|c|c|}
\hline
$H_1\colon\left(\begin{matrix}*&0\\ 0&I\end{matrix}\right)$&
$H_2\colon\left(\begin{matrix}*&*\\ 0&I\end{matrix}\right)$&
$H_3\colon\left(\begin{matrix}*&*\\ 0&*\end{matrix}\right)$&
$H_4\colon\left(\begin{matrix}*&0\\ 0&*\end{matrix}\right)$\\
&&&\\
$p\leq q+1$&
$p= 1$&
$p=q=1$&
$p\leq q+1$\\
&&&
$q\leq p+1$\\
\hline
\end{tabular}
\end{center}
\vspace{-1em}
{\caption{Table 1:\;
The criterion $\rho_{\gs h}\leq\rho_{\gs g/\gs h}$
when $G= {\rm SL}(p+q,\m R)$}}
\label{figslpq}
\end{figure}

In Table 2, we specify our criterion \eqref{eqnlghtem} when
$
G= {\rm SL}(\m R^p\oplus\m R^q\oplus\m R^r) 
$
and  $H$ is a  subgroup of $G$  
normalized by the 
group ${\rm SL}(\m R^p)\times {\rm SL}(\m R^q) \times {\rm SL}(\m R^r) $.
Note that in these two tables, the center of the diagonal blocks is not important by Corollary \ref{corgh1h2}. 
\begin{figure}[htb]
\begin{center}
\begin{tabular}{|c|c|c|c|}
\hline
$H_1:\left(\begin{matrix}*&0&*\\ 0&I&0\\ 0&0&I\end{matrix}\right)$&
$H_2:\left(\begin{matrix}I&0&*\\ 0&*&0\\ 0&0&I\end{matrix}\right)$&
$H_3:\left(\begin{matrix}I&*&*\\ 0&*&0\\ 0&0&I\end{matrix}\right)$&
$H_4:\left(\begin{matrix}*&*&*\\ 0&*&*\\ 0&0&*\end{matrix}\right)$\\
&&&\\
$p\leq q+1$&
$q\leq p+r+1$&
$q\leq r+1$ &
$p= q=r=1$\\
\hline
\hline 
$H_5:\left(\begin{matrix}*&0&0\\ 0&*&0\\ 0&0&I\end{matrix}\right)$&
$H_6:\left(\begin{matrix}*&0&*\\ 0&*&0\\ 0&0&I\end{matrix}\right)$&
$H_7:\left(\begin{matrix}*&*&*\\ 0&*&0\\ 0&0&I\end{matrix}\right)$&
$H_8:\left(\begin{matrix}*&0&*\\ 0&I&0\\ 0&0&*\end{matrix}\right)$\\
&&&\\
$p\leq q+r+1$ &
$p\leq q+1$&
$p=1$&
$p\leq q+1$\\
$q\leq p+r+1$&
$q\leq p+r+1$&
$q\leq r+1$&
$r\leq q+1$\\
\hline
\hline
$H_9\!:\left(\begin{matrix}I&*&*\\ 0&*&0\\ 0&0&*\end{matrix}\right)$&
$H_{10}\!:\left(\begin{matrix}*&0&0\\ 0&*&0\\ 0&0&*\end{matrix}\right)$&
$H_{11}\!:\left(\begin{matrix}*&0&*\\ 0&*&0\\ 0&0&*\end{matrix}\right)$&
$H_{12}\!:\left(\begin{matrix}*&*&*\\ 0&*&0\\ 0&0&*\end{matrix}\right)$\\
&&&\\
$q\leq r+1$ &
$p\leq q+r+1$&
$p\leq q+1$&
$p=1$\\
$r\leq q+1$&
$q\leq p+r+1$&
$q\leq p+r+1$&
$q\leq r+1$\\
&
$r\leq p+q+1$&
$r\leq q+1$&
$r\leq q+1$\\
\hline
\end{tabular}
\end{center}
\vspace{-1em}
\caption{Table 2:\;
The criterion $\rho_{\gs h}\leq\rho_{\gs g/\gs h}$
when $G= {\rm SL}(p+q+r,\m R)$}
\label{figslpqr}
\end{figure}
\begin{Rem}
It is rather easy to guess the inequalities
in Table \ref{figslpqr}.
Here is the heuristic recipe: there is one inequality
for each non-identity diagonal block. 
The left-hand side of this inequality is given by the size
of this diagonal block, while
the right-hand side can be guessed by looking at 
the size of the zero blocks on the right and on the top of it. 
\end{Rem}

We will just explain the proof for the group $H=H_{11}$
in Table \ref{figslpqr}.
The other cases are  similar. 

\begin{Cor}
\label{corslpqr}
Let $G= {\rm SL}(p\!+\!q\!+\!r,\m R)$ and 
$ H$  the subgroup of matrices 
$\left(\begin{matrix}\al &0&z\\ 0&\be&0\\ 0&0&\ga\end{matrix}\right)$ 
with $\al\in {\rm GL}(p,\m R)$, 
$\be\in {\rm GL}(q,\m R)$, $\ga\in {\rm GL}(r,\m R)$,
$z\in {\rm M}(p,r;\m R)$.
Then $L^2(G/H)$ is $G$-tempered if and only if
$\;
p\leq q\!+\!1,\;
q\leq p\!+\!r\!+\!1,\;
r\leq q\!+\!1.
$
\end{Cor}

\begin{proof}[Proof of Corollary \ref{corslpqr}]
We denote by $\g a$ the Lie algebra of diagonal matrices
$$
\g a=\{Y=(x,y,z)\in \m R^p\oplus\m R^q\oplus\m R^r\mid {\Tr}(Y)=0\}.
$$
We only need to check the criterion 
\eqref{eqnlghtem} on the chamber
$$
\g a_+=\{Y=(x,y,z)\in \g a
\mid \mbox{ $x$, $y$ and $z$ have non-decreasing coordinates}\}.
$$
We recall that $\g q=\g g/\g h$ and 
we compute for $Y\in \g a_+$,
$$
\rho_\gs h(Y)=\textstyle
\sum_{i=1}^pa_ix_i+
\sum_{j=1}^qb_jy_j+
\sum_{k=1}^rc_kz_k
\;,
$$
where $\;\;a_i:= 2i-p-1,\;\;
b_j:= 2j-q-1,\;\;
c_k:=2k-r-1$, \; and
$$\textstyle
\rho_\gs q(Y)=
\sum_{i,j}|x_i-y_j| +\sum_{j,k}|y_j-z_k|\; .
$$

Assume first that the criterion $\rho_\gs h \leq \rho_\gs q $ is satisfied on $\g a$. It is then also satisfied on $\m R^{p+q+r}$. 
Applying it successively to the three vectors $Y=e_p$, $Y=e_{p+q}$ and 
$Y=e_{p+q+r}$ of the standard basis $e_1,\ldots ,e_{p+q+r}$ of $\m R^{p+q+r}$, 
one gets successively the three inequalities
$\;
p\leq q\!+\!1,\;
q\leq p\!+\!r\!+\!1,\;
r\leq q\!+\!1.
$
\vs

Assume now that these three inequalities are satisfied.
Note that
\begin{eqnarray*}
\rho_\gs q(Y)\!\!&\geq&\!\! \textstyle
\sum_{a_i>b_j}\! (x_i\!-\! y_j) +
\sum_{b_j>a_i}\! (y_j\!-\! x_i)  +
\sum_{b_j>c_k}\! (y_j\!-\! z_k)  +
\sum_{c_k>b_j}\! (z_k\!-\! y_j)   \\
&=&\textstyle
\sum_{i=1}^p \ell_i x_i +
\sum_{j=1}^q m_j  y_j +
\sum_{k=1}^r n_k z_k, 
\;\; {\rm where} 
\end{eqnarray*}
\begin{eqnarray*}
\ell_i&=&
|\{j\mid b_j<a_i\}|-|\{j\mid b_j>a_i\}|\\
m_j&=&
|\{i\mid a_i<b_j\}|-|\{i\mid a_i>b_j\}| 
+|\{k\mid c_k<b_j\}|-|\{k\mid c_k>b_j\}|\\ 
n_k&=&
|\{j\mid b_j<c_k\}|-|\{j\mid b_j>c_k\}|.
\end{eqnarray*}

Since $p\leq q+1$, one has $\ell_i=a_i$ for all $1\leq i\leq p$.

Since $r\leq q+1$, one has $n_k=c_k$ for all $1\leq k\leq r$.

For $1\leq j\leq q$, one has
$m_{q+1-j}=-m_j$ and, when $j> q/2$,
$$
m_j= \operatorname{min}(b_j,p)+\operatorname{min}(b_j,r).
$$
Since $q\leq p+r+1$, one has 
$m_j\geq b_j$ for all $j> q/2$. 
Then using the fact that 
the $y_j$s are non-decreasing functions of $j$, one gets,
for $Y$ in $\g a_+$,
$$\textstyle
\rho_\gs q(Y)-\rho_{\gs h}(Y)=\sum_{j=1}^q(m_j-b_j)y_j\geq 0.
$$
This proves that the criterion $\rho_\gs h \leq \rho_\gs q $ is satisfied.
\end{proof}

Some of the subgroups in Table \ref{figslpqr} appear naturally
 in analyzing the tensor product representations
 of $SL(n,{\mathbb{R}})$ as below.  

\subsection{Tensor product of non-tempered representations}
\label{sectensor}

Suppose $\Pi$ and $\Pi'$ are unitary representations of $G$.  
The tensor product representation $\Pi \otimes \Pi'$ is tempered
 if $\Pi$ or $\Pi'$ is tempered.  
In contrast, $\Pi \otimes \Pi'$ may be and may not be tempered 
 when both $\Pi$ and $\Pi'$ are non-tempered.  

For instance, let $n=n_1 + \cdots + n_k$ be a partition,
 and we consider the (degenerate) principal series representation
 $\Pi_{n_1,\cdots,n_k}:=\Ind_{P_{n_1,\cdots,n_k}}^{G}({\bf{1}})$
 of $G=SL(n,{\mathbb{R}})$, 
 where $P_{n_1,\cdots,n_k}$ is the standard parabolic subgroup 
 with Levi subgroup $S(GL(n_1,{\mathbb{R}}) \times \cdots \times GL(n_k,{\mathbb{R}}))$.  
Then $\Pi_{n_1,\cdots,n_k}$ is tempered 
 iff $k=n$ and $n_1= \cdots = n_k=1$.  
Here are some examples of the temperedness criterioin \eqref{eqnlghtem}
 applied to the tensor product of two such representations.  
\begin{Prop}
Let $0 \le k, l \le n$ and $a+b+c =n$.  
\begin{enumerate}
\item[{\rm{(1)}}]
$\Pi_{k,n-k} \otimes \Pi_{n-l,l}$ is tempered
 iff $|k-l| \le 1$ and $|k+l-n| \le 1$.  
\item[{\rm{(2)}}]
$\Pi_{a,b,c} \otimes \Pi_{b+c,a}$ is tempered
 iff $\max(b,c) -1 \le a \le b+c+1$.  
\item[{\rm{(3)}}]
$\Pi_{a,b,c} \otimes \Pi_{c,b,a}$ is tempered
 iff $2 \max(a,b,c) \le n+1$. 
\end{enumerate}
\end{Prop}

\begin{proof}
For any parabolic subgroups $P$ and $P'$ of $G$, there
exists an element $w \in G$ such that $P w P'$ is open dense in $G$,
and thus the tensor product $\Ind_{P}^G ({\bf 1}) \otimes \Ind_{P'}^G ({\bf 1})$
is unitarily equivalent to the regular representation in $L^2(G/H)$ 
by the Mackey theory, where $H=w^{-1}P w \cap P'$.
In the above cases, we have the following unitary equivalences:
\begin{alignat*}{2}
\Pi_{k,n-k} \otimes \Pi_{n-l,l} &\simeq L^2(G/H_{12})
\quad
&& 
\!\!\!\text{with } (p,q,r) =\!(|k\!-\!l|, \min(k,l), 
n\!-\!\max(k,l)), 
\\
\Pi_{a,b,c} \otimes \Pi_{b+c,a} &\simeq L^2(G/H_{11})
\quad
&& 
\!\!\!\text{with } (p,q,r)=(b,a,c), 
\\
\Pi_{a,b,c} \otimes \Pi_{c,b,a} &\simeq L^2(G/H_{10})
\quad
&& 
\!\!\!\text{with } (p,q,r)=(a,b,c),
\end{alignat*}
whence Proposition follows from Table \ref{figslpqr} in Section \ref{secnonred}.  
\end{proof}

\subsection*{Acknowledgments}

The authors are grateful to the IHES and to the University of Tokyo for its support through the GCOE program.
The second author was partially supported by Grant-in-Aid for Scientific Research 
(A) (25247006) JSPS.

{\small
\bibliography{temperedII}
}

{\small
\noindent
Y. \textsc{Benoist} :
CNRS-Universit\'e Paris-Saclay, Orsay, France\newline
e-mail : \texttt{yves.benoist@math.u-psud.fr}

\medskip
\noindent
T. \textsc{Kobayashi} :
Graduate School of Mathematical Sciences and 
Kavli IPMU-University of Tokyo, Komaba,  Japan\newline
e-mail : \texttt{toshi@ms.u-tokyo.ac.jp}
\end{document}